\documentclass[12pt]{amsart}
\usepackage{import}
\usepackage{mypreamble}
\usepackage{subfiles}

\title{Weighted 1-dimensional Orlicz-Poincar\'e inequalities}
\author{Lyudmila Korobenko}
\address{Department of Mathematics \\
          Reed College \\
         Portland, OR 97202}
\email{korobenko@reed.edu}
\author{Olly Milshstein}
\address{Department of Mathematics \\
          Reed College \\
         Portland, OR 97202}
\email{milshteo@reed.edu}
\author{Lucas Yong}
\address{Department of Mathematics \\
          University of Oklahoma \\
          Norman, OK 73019}
\email{lucasyong@ou.edu}

\begin{document}
\maketitle
\tableofcontents
\begin{abstract}
In this paper we establish necessary and sufficient conditions for weighted Orlicz-Poincar\'{e} inequalities in dimension one. Our theorems generalize the main results of Chua and Wheeden \cite{chuawheeden}, who established necessary and sufficient conditions for weighted $(q,p)$ Poincar\'{e} inequalities. We give an example of a weight satisfying sufficient conditions for a $(\Phi,p)$ Orlicz-Poincar\'{e} inequality where the gauge norm with respect to $\Phi$ is a bump on the Lebesgue $L^p$ norm. This weight, on the other hand, does not satisfy a $(q,p)$ Poincar\'{e} inequality for any $q>p$.
\end{abstract}
\section{Introduction}

Sobolev spaces on metric measure spaces have been extensively studied \cite{Haj, Haj2, Haj3, Shan, Durand, AlvHaj2, Heik2}. However, there are few easily verifiable conditions on metric measure spaces which are sufficient for such inequalities to hold for a set of test functions (such as Lipschitz functions), see for example \cite{saloff-coste_2001, hinpet,chuawheeden}. On the other hand, Sobolev-Poincar\'{e} inequalities and even sufficiently strong Orlicz-Sobolev inequalities imply certain properties of the underlying metric measure spaces, such as the doubling condition \cite{Haj2, KoMaRi, AlvHaj2, Heik2}.

A particular interest in Sobolev and Poincar\'{e} inequalities arises from the regularity theory for degenerate elliptic operators. More precisely, given a degenerate elliptic operator with a degeneracy controlled by a weight, one can use weighted Sobolev and Poincar\'{e} inequalities to perform the Moser or DeGiorgi iteration scheme to show regularity of weak solutions \cite{FKS}. There is another class of degeneracies, to which one can associate a control metric, and follow a somewhat similar iteration scheme using the theory of Sobolev spaces on metric spaces \cite{FL}. It has also been shown recently that a weaker Orlicz-Sobolev inequality can be used in a modification of the DeGiorgi scheme to show continuity of weak solutions to infinitely degenerate elliptic equations \cite{KRSSh}. This leads to a natural question: what conditions on a metric measure space are sufficient (and necessary) for it to support different versions of Sobolev type inequalities? This is the main question that we partially address in this paper.

More precisely we follow the work of Chua and Wheeden \cite{chuawheeden} to give necessary and sufficient conditions for a one-dimensional Orlicz-Poincar\'{e} inequality, see Theorems \ref{thm:CW1=p}, \ref{thm:CW1<p_forward}, and \ref{thm:CW1<p_backward}. \autoref{thm:CW1=p} provides a necessary and sufficient condition for a $(\Phi, 1)$ Orlicz-Poincar\'{e} inequality, where the gauge norm with respect to $\Phi$ is a bump on the $L^1$ norm. When $p>1$, our necessary condition, Theorem \ref{thm:CW1<p_backward}, just misses to be sufficient. The mismatch comes from the fact that Young functions that appear in Orlicz-Poincar\'{e} inequalities are not multiplicative, unlike power functions appearing in the classical Sobolev and Poincar\'{e} inequalities. We do not know if the gap between our necessary and sufficient conditions can be closed, but we suspect that a different method has to be used. 
Our arguments follow closely those in \cite{chuawheeden} and rely on the clever use of the fundamental theorem of calculus, Minkowski's integral inequality, and Hardy type inequalities. However, some of the tools available for Lebesgue norms, such as Minkowski and Hardy type inequalities, need to be adapted to the setting of Orlicz spaces; see Lemmas \ref{lemma:gen_minkowski} and \ref{lemma:opickufneradaptation1}, for example. This presents some technical difficulties, since we are no longer working with power functions, which are multiplicative, but rather with submultiplicative Young functions. This makes our main results interesting and nontrivial. We now describe them in more detail.

Let $-\infty < a < b < \infty$, $f\colon [a,b] \to \R$ a Lipschitz continuous function, $\mu$ and $w$ weights on $[a,b]$, and $\nu$ a nonnegative finite Borel measure on $[a,b]$ satisfying $\nu[a,b]  > 0$. In \cite{chuawheeden}, the authors obtain necessary and sufficient conditions for the inequality
\begin{equation}
\label{eqn:chuawheedenmain}
\left\|f - \frac{1}{\nu[a,b]}\int_a^b f \dd \nu\right\|_{L_\mu^q[a,b]}
\leq
C \left\|f'\right\|_{L_w^p[a,b]},
\end{equation}
where $C$ is a constant. That is, the inequality above is valid if and only if a special constant $K_{p,q}(\mu,\nu, w)$ is finite. The value of $K_{p,q}(\mu,\nu, w)$ varies depending on the values of $p$ and $q$; we refer the reader to \cite[Theorem 1.4]{chuawheeden} for the details.

We generalize the results in \cite{chuawheeden} by replacing the $L^q$ norm on the lefthand side of \autoref{eqn:chuawheedenmain} with the gauge norm with respect to a Young function $\Phi$, see Section \ref{sec:prelim}. The assumptions on $a,b\in\R$ and measures $\mu,\nu,w$ are as above and are the same throughout the paper unless stated otherwise. Assumptions on the Young function $\Phi$ vary depending the value of $p$ and other details, and hence will be stated in each of the main results, and the respective sections devoted to their proofs.

Our main results are as follows.
\begin{theorem}
\label{thm:CW1=p}
Define 
\[
K_{1,\Phi}(\mu, \nu, w) := 
    \frac{1}{\nu[a,b]}
    \sup_{a<x<b}\left\{
    \frac{1}{w(x)}
    \left\|
    \nu[a,x]\chi_{[x,b]}
    -
    \nu[x,b]\chi_{[a,x]}
    \right\|_{L_\mu^\Phi[a,b]}
    \right\}.
    \]
Then $K_{1,\Phi}(\mu, \nu, w) < \infty$ if and only if
\begin{equation}
\label{eqn:CW1=p}
\left\|f - \frac{1}{\nu[a,b]}\int_a^b f \dd \nu\right\|_{L_\mu^\Phi[a,b]}
\leq
C
\left\|f'\right\|_{L_w^1[a,b]}
\end{equation}
for all Lipschitz continuous functions $f$, and for some constant $C > 0$. Moreover, $K_{1,\Phi}(\mu,\nu,w) \leq C$ for any such constant.
\end{theorem}

\begin{theorem}
\label{thm:CW1<p_forward}
Let $1 < p  < \infty$. Assume that $\Phi$ is submultiplicative and invertible on $[0,\infty)$, and that that the function $\Lambda$ defined by $\Lambda(t) = \Phi\left(t^{\frac{1}{p}}\right)$ is convex.

Define
    \[
    \begin{aligned}
    K_{p,\Phi}(\mu, \nu, w) := 
    \frac{1}{\nu[a,b]}
    &\Biggl(
    \sup_{a<x<b}\biggl\{
    \left[\Phi^{-1}\left(\frac{1}{\mu[x,b]^{1/2}}\right)\right]^{-2}\biggl(\int_a^x \nu[a,t]^{p'}w(t)^{1-p'} \dd t\biggr)^{1/p'}
    \biggr\}\\
    &+
    \sup_{a<x<b}\biggl\{
    \left[\Phi^{-1}\left(\frac{1}{\mu[a,x]^{1/2}}\right)\right]^{-2}\biggl(
    \int_x^b \nu[t,b]^{p'} w(t)^{1-p'} \dd t
    \biggr)^{1/p'}
    \biggr\}
    \Biggr).
    \end{aligned}
    \]
Assume that $K_{p,\Phi}(\mu, \nu, w) < \infty$. Then,
\begin{equation}
\label{eqn:CW1<p}
\left\|f - \frac{1}{\nu[a,b]}\int_a^b f \;\dd\nu\right\|_{L_\mu^\Phi[a,b]} 
\leq 
C
\|f'\|_{L_w^p[a,b]},
\end{equation}
for all Lipschitz continuous functions $f$, and for some constant $C \leq C_0(\Phi) K_{p,\Phi}(\mu,\nu,w)$, where
\[
C_0(\Phi) := 2
\left[\Phi^{-1}\left(\frac{1}{2}\right)\right]^{-1}.
\]
\end{theorem}

\begin{theorem}
\label{thm:CW1<p_backward}
Assume that $\Phi$ is invertible on $[0,\infty)$. Define
\[
\begin{aligned}
\tilde{K}_{p,\Phi}(\mu, \nu, w) := 
\frac{1}{\nu[a,b]}
&\Biggl(
\sup_{a<x<b}\biggl\{
\left[\Phi^{-1}\left(\frac{1}{\mu[x,b]}\right)\right]^{-1}\biggl(\int_a^x \nu[a,t]^{p'}w(t)^{1-p'} \dd t\biggr)^{1/p'}\biggr\}\\
&+
\sup_{a<x<b}\biggl\{
\left[\Phi^{-1}\left(\frac{1}{\mu[a,x]}\right)\right]^{-1}\biggl(
\int_x^b \nu[t,b]^{p'} w(t)^{1-p'} \dd t
\biggr)^{1/p'}
\biggr\}
\Biggr).
\end{aligned}
\]
Assume that
\[
\left\|f - \frac{1}{\nu[a,b]}\int_a^b f \dd \nu\right\|_{L_\mu^\Phi[a,b]}
\leq
C 
\left\|f'\right\|_{L_w^p[a,b]},
\]
for all Lipschitz continuous functions $f$ and for some constant $C$. Then, $\tilde{K}_{p,\Phi}(\mu, \nu, w) < \infty$.
\end{theorem}

\begin{remark}
The reader should compare the constant $K_{p,\Phi}(\mu, \nu, w)$ in \autoref{thm:CW1<p_forward} to the constant $\tilde{K}_{p,\Phi}(\mu, \nu, w)$ in \autoref{thm:CW1<p_backward}. We remark that it is always true that 
\[
K_{p,\Phi}(\mu, \nu, w) \geq \tilde{K}_{p,\Phi}(\mu, \nu, w),
\]
with equality in certain special cases; for example when we choose $\Phi(t) = |t|^q$ for $p \leq q$ (for this choice of $\Phi$, $K_{p,\Phi} = \tilde{K}_{p,\Phi}$ is equal to the constant in \cite[Theorem 1.4]{chuawheeden}). At the time of writing, we are uncertain about whether the gap between the two constants can be closed for arbitrary $\Phi$, but we suspect that different methods must be used to achieve this result.
\end{remark}
This paper is organized as follows. After some preliminaries in Section \ref{sec:prelim}, we prove necessary and sufficient conditions for the inequality \autoref{eqn:CW1=p}, (i.e., the case when $p=1$) in Section \ref{sec:p=1}. This closely follows the proof of \cite[Theorem 1.4]{chuawheeden}. Sections \ref{sec:p-suff} and \ref{sec:p-necc} are concerned with the proofs of \autoref{thm:CW1<p_forward} and \autoref{thm:CW1<p_backward} (i.e., when $p > 1$); these again employ the techniques used in \cite{chuawheeden}, and also use ideas from \cite{opickufner}. Finally, in Section 6, we provide an example of a weight $w$ where the gauge norm is a bump on the the Lesbesgue $L^p$ norm, but which does not satisfy a $(q,p)$ Poincar\'e inequality for any $q > p$.
\section{Preliminaries}
\label{sec:prelim}
In this section, we review some important aspects of the theory of Orlicz spaces. We refer the reader to \cite{raoren} for a comprehensive introduction.
\begin{defn}
A \textbf{Young function} is a convex function $\Phi\colon \R \to [0, \infty]$ such that
\begin{enumerate}[label=(\roman*)]
    \item $\Phi$ is even, i.e. $\Phi(-t) = \Phi(t)$,
    \item $\Phi(0) = 0$,
    \item $\lim_{t \to \infty} \Phi(t) = +\infty$.
\end{enumerate}
\end{defn}
\begin{defn}
Given a Young function $\Phi$, the \textbf{complementary function} to $\Phi$ is another convex function $\Psi \colon \R \to [0,\infty]$, defined by
\[
\Psi(s) = \sup\{t|s| - \Phi(t) : t \geq 0\}.
\]
\end{defn}
The reader may verify that $\Psi$ is itself a Young function, and so the pair $(\Phi, \Psi)$ may be called a complementary pair (of Young functions).
For the remainder of this section, let $(\Phi, \Psi)$ be a complementary pair of Young functions, and let $(\Omega, \Sigma, \mu)$ be an arbitrary measure space. 
\begin{defn}
Define
\[
L_\mu^\Phi := \left\{f \colon \Omega \to \overline{\R}: f \text{ is measurable and } \int_\Omega \Phi(\alpha f) \dd\mu < \infty \text{ for some } \alpha > 0\right\}.
\]
\end{defn}
The set $L_\mu^\Phi$ is a vector space. We now define a norm on $L_\mu^\Phi$.
\begin{defn}
Let $f \in L_\mu^\Phi$, and define $\|~\|_{L_\mu^\Phi} \colon L_\mu^\Phi \to [0,\infty]$ by
\[
\|f\|_{L_\mu^\Phi} := \inf\left\{k > 0 : \int_\Omega \Phi\left(\frac{f}{k}\right)\dd \mu \leq 1\right\}.
\]
$\|~\|_{L_\mu^\Phi}$ is called the \textbf{gauge norm} (or \emph{Luxemburg} norm).
\end{defn}
The gauge norm makes $L_\mu^\Phi$ a Banach space; see \cite{raoren} for details. Note that we may also define $L_\mu^\Phi$ in terms of this norm, i.e.
\[
L_\mu^\Phi = \left\{f \colon \Omega \to \overline{\R}: f \text{ is measurable and } \|f\|_{L_\mu^\Phi} < \infty\right\}.
\]
\begin{example}
Let $\Phi(t) = |t|^p$, where $p \geq 1$. Then
\[
\begin{aligned}
\|f\|_{L_\mu^\Phi} &= \inf\left\{k > 0 : \int_\Omega \Phi\left(\frac{f}{k}\right)\dd \mu \leq 1\right\}\\
&= \inf\left\{k > 0 : \int_\Omega \frac{|f|^p}{k^p}\dd \mu \leq 1\right\}\\
&= \inf\left\{k > 0 : \int_\Omega |f|^p\dd \mu \leq k^p\right\}\\
&= \inf\left\{k > 0 : \left(\int_\Omega |f|^p\dd \mu\right)^{1/p} \leq k\right\}\\
&= \left(\int_\Omega |f|^p\dd \mu\right)^{1/p},
\end{aligned}
\]
recovering the $L^p$-norm. Thus, for this choice of $\Phi$, we have that $L_\mu^\Phi = L_\mu^p$. In this sense, Orlicz spaces are a generalization of $L^p$ spaces.
\end{example}
\begin{defn}
Let $(\Phi, \Psi)$ be a complementary Young pair. The \textbf{Orlicz norm} is
\[
|f|_{L_\mu^\Phi} := \sup_{g \in L_\mu^\Psi}\left\{\int_\Omega |fg| \dd \mu : \|g\|_{L_\mu^\Psi} \leq 1\right\}.
\]
\end{defn}
In fact, the gauge and Orlicz norms are equivalent, a fact that we state without proof (see \cite[Proposition 3.3.4]{raoren}).
\begin{prop}
\label{prop:equivnorms}
For any $f \in L_\mu^\Phi$,
\[
\|f\|_{L_\mu^\Phi} \leq |f|_{L_\mu^\Phi} \leq 2\|f\|_{L_\mu^\Phi}.
\]
\end{prop}
We conclude this section with a generalization of Minkowski's inequality for integrals, which will be important in the proofs contained in the next section.
\begin{lemma}
\label{lemma:gen_minkowski}
Let $F \colon \R \times \R \to \R$ be a measurable function. Then,
\[
\left\|\int F(\bullet,t) \dd t\right\|_{L_\mu^\Phi} 
\leq 2\int \left\|F(\bullet,t)\right\|_{L_\mu^\Phi}  \dd t
\]
\end{lemma}
\begin{proof}
Let $\Psi$ be the complementary function to $\Phi$. Using \autoref{prop:equivnorms} and Fubini's theorem, we have
\[
\begin{aligned}
\left\|\int F(\bullet,t) \dd t\right\|_{L_\mu^\Phi} 
&\leq \left|\int F(\bullet,t) \dd t\right|_{L_\mu^\Phi}\\
&= \sup_{\substack{g \in L_\mu^{\Psi};\\ \left\|g\right\|_{L_\mu^\Psi} \leq 1}}\left(\iint F(x,t) g(x) \mu(x) \dd t \dd x\right)\\
&= \sup_{\substack{g \in L_\mu^{\Psi};\\ \left\|g\right\|_{L_\mu^\Psi} \leq 1}}\left(\iint F(x,t) g(x) \mu(x) \dd x \dd t\right)\\
&\leq \int\sup_{\substack{g \in L_\mu^{\Psi};\\ \left\|g\right\|_{L_\mu^\Psi} \leq 1}}\left(\int F(x,t) g(x) \mu(x) \dd x \right)\dd t\\
&= \int \left|F(\bullet,t)\right|_{L_\mu^\Phi} \dd t\\
&\leq 2\int \left\|F(\bullet,t)\right\|_{L_\mu^\Phi}  \dd t,
\end{aligned}
\]
as desired.
\end{proof}
\section{Necessity and sufficiency when \texorpdfstring{$p=1$}{p=1}}\label{sec:p=1}
We prove \autoref{thm:CW1=p} by modifying the argument in the proof of \cite[Theorem 1.4]{chuawheeden}.
\begin{proof}
(of \autoref{thm:CW1=p}) Assume that $K_{1,\Phi}(\mu,\nu,w) < \infty$. By \cite[Equation 2.1]{chuawheeden},
\[
f(x) - \frac{1}{\nu[a,b]}\int_a^b f \dd \nu 
=
\frac{1}{\nu[a,b]}
\left(\int_a^b \nu[a,z] f'(z) \chi_{[a,x]}(z) - \nu[z,b] f'(z) \chi_{[x,b]}(z) \dd z\right).
\]
for all $x \in [a,b]$. Using \autoref{lemma:gen_minkowski}, we have
\[
\begin{aligned}
\biggl\| f - \frac{1}{\nu[a,b]}\int_a^b f \dd \nu  \biggr\|_{L_\mu^\Phi[a,b]}
&=
\frac{1}{\nu[a,b]}\biggl
\| \int_a^b \nu[a,z] f'(z) \chi_{[a,\bullet]}(z) - \nu[z,b] f'(z) \chi_{[\bullet,b]}(z) \dd z\biggr\|_{L_\mu^\Phi[a,b]}\\
&\leq
\frac{2}{\nu[a,b]}\int_a^b\biggl
\|f'(z)\left(\nu[a,z]\chi_{[a,\bullet]}(z) - \nu[z,b] \chi_{[\bullet,b]}(z)\right) \biggr\|_{L_\mu^\Phi[a,b]}\dd z\\
&= 
\frac{2}{\nu[a,b]}\int_a^b\biggl
\|\nu[a,z]\chi_{[a,\bullet]}(z) - \nu[z,b] \chi_{[\bullet,b]}(z) \biggr\|_{L_\mu^\Phi[a,b]} |f'(z)| \dd z\\
&= 
\frac{2}{\nu[a,b]}\int_a^b\biggl
\|\nu[a,z]\chi_{[z,b]} - \nu[z,b] \chi_{[a,z]} \biggr\|_{L_\mu^\Phi[a,b]} |f'(z)|\dd z\\
&= 
\frac{2}{\nu[a,b]}
\int_a^b \frac{1}{w(z)}
\biggl\|\nu[a,z]\chi_{[z,b]} - \nu[z,b] \chi_{[a,z]} \biggr\|_{L_\mu^\Phi[a,b]} 
|f'(z)| w(z) \dd z\\
&\leq 
\frac{2}{\nu[a,b]}
\int_a^b \sup_{a<x<b}\left\{\frac{1}{w(x)}
\biggl\|\nu[a,x]\chi_{[x,b]} - \nu[x,b] \chi_{[a,x]} \biggr\|_{L_\mu^\Phi[a,b]}\right\} 
|f'(z)| w(z) \dd z\\
&= 2K_{1,\Phi}(\mu,\nu,w) \int_a^b |f'(z)| w(z) \dd z\\
&= 2K_{1,\Phi}(\mu,\nu,w) \left\|f'\right\|_{L_w^1[a,b]}.
\end{aligned}
\]

\vspace{+0.2in}
Conversely, assume that there exists $C > 0$ such that \autoref{eqn:CW1=p} holds for all Lipschitz continuous functions $f \colon [a,b] \to \R$.

For all $n\in\Z^+$, define $w_n(x) = w(x) + \frac{1}{n}$. Let $\alpha$ be a Lebesgue point of $\frac{1}{w_n}$ such that $a<\alpha<b$ and $\nu[a,\alpha] + \nu[\alpha,b] = \nu[a,b]$. Note that the latter equality holds except for at most countably many $\alpha$, and the Lebesgue differentiation theorem guarantees that almost all $\alpha$ are Lebesgue points. 

Next, for any $0 < \varepsilon < \min\{\alpha-a, b-\alpha\}$, define
\[
f_{\varepsilon}(x) = 
\begin{cases}
-\int_x^\alpha \frac{\nu[\alpha,b]}{\nu[a,\alpha]}\frac{\chi_{[\alpha-\varepsilon, \alpha]}(t)}{\varepsilon w_n(t)} \dd t,
&\text{if $a\leq x \leq \alpha$;}\\
\int_\alpha^x \frac{\chi_{[\alpha, \alpha+\varepsilon]}(t)}{\varepsilon w_n(t)} \dd t,
&\text{if $\alpha < x \leq b$.}\\
\end{cases}
\]
We claim that
\begin{equation}
\label{eqn:p=1helper}
\lim_{\varepsilon \to 0}
\left\|f_\varepsilon - \frac{1}{\nu[a,b]} \int_a^b f_\varepsilon \dd \nu \right\|_{L_\mu^\Phi[a,b]}
=
\frac{1}{w_n(\alpha)} 
\left\|\frac{\nu[a,b]}{\nu[a,\alpha]} \chi_{[a,\alpha]} + \chi_{[\alpha,b]} \right\|_{L_\mu^\Phi[a,b]}.
\end{equation}
First note that, using the fact that $\alpha$ is a Lebesgue point of $\frac{1}{w_n}$, we have that as $\varepsilon \to 0$,
\[
\frac{1}{\nu[a,b]} \int_a^b f_\varepsilon \dd \nu \to 0.
\]
Thus, we need only consider $\lim_{\varepsilon \to 0}
\left\|f_\varepsilon \right\|_{L_\mu^\Phi[a,b]}$ in order to prove \autoref{eqn:p=1helper}.
By definition,
\[
\left\|f_\varepsilon\right\|_{L_\mu^\Phi[a,b]}
=
\inf\left\{k > 0 : \int_a^b\Phi\left(\frac{f_\varepsilon}{k}\right) \dd \mu \leq 1\right\}.
\]
Notice that
\[
\begin{aligned}
\int_a^b\Phi\left(\frac{f_\varepsilon(x)}{k}\right) \mu(x) \dd x
=& \int_a^\alpha\Phi\left(\frac{f_\varepsilon(x)}{k}\right) \mu(x) \dd x
+ \int_\alpha^b\Phi\left(\frac{f_\varepsilon(x)}{k}\right) \mu(x) \dd x\\
=& \int_a^\alpha\Phi\left(\frac{1}{k}\int_x^\alpha \frac{\nu[\alpha,b]}{\nu[a,\alpha]}\frac{\chi_{[\alpha-\varepsilon, \alpha]}(t)}{\varepsilon w_n(t)} \dd t\right) \mu(x) \dd x\\
&+ \int_\alpha^b\Phi\left(\frac{1}{k}\int_\alpha^x \frac{\chi_{[\alpha, \alpha+\varepsilon]}(t)}{\varepsilon w_n(t)} \dd t\right) \mu(x) \dd x\\
=& \int_a^{\alpha-\varepsilon}\Phi\left(\frac{1}{k}\int_{\alpha-\varepsilon}^\alpha \frac{\nu[\alpha,b]}{\nu[a,\alpha]}\frac{1}{\varepsilon w_n(t)} \dd t\right) \mu(x) \dd x\\
&+ \int_{\alpha-\varepsilon}^\alpha\Phi\left(\frac{1}{k}\int_x^\alpha \frac{\nu[\alpha,b]}{\nu[a,\alpha]}\frac{\chi_{[\alpha-\varepsilon, \alpha]}(t)}{\varepsilon w_n(t)} \dd t\right) \mu(x) \dd x\\
&+ \int_\alpha^{\alpha + \varepsilon}\Phi\left(\frac{1}{k}\int_\alpha^x \frac{\chi_{[\alpha, \alpha+\varepsilon]}(t)}{\varepsilon w_n(t)} \dd t\right) \mu(x) \dd x\\
&+ \int_{\alpha + \varepsilon}^b\Phi\left(\frac{1}{k}\int_\alpha^{\alpha+\varepsilon} \frac{1}{\varepsilon w_n(t)} \dd t\right) \mu(x) \dd x.
\end{aligned}
\]
The second and third terms on the righthand side vanish as $\varepsilon \to 0$. To see this, recall that $\alpha$ is a Lebesgue point of $\frac{1}{w_n}$, and hence, as $\varepsilon \to 0$,
\[
\left(\int_x^\alpha \frac{\chi_{[\alpha-\varepsilon,\alpha]}(t)}{\varepsilon w_n(t)}\right)
\leq
\frac{1}{\varepsilon} \int_{\alpha-\varepsilon}^\alpha \frac{1}{w_n(t)} \dd t \longrightarrow \frac{1}{w_n(\alpha)},
\]
and similarly
\[
\left(\int_\alpha^x \frac{\chi_{[\alpha,\alpha+\varepsilon]}(t)}{\varepsilon w_n(t)}\right)
\leq
\frac{1}{\varepsilon} \int_\alpha^{\alpha+\varepsilon} \frac{1}{w_n(t)} \dd t \longrightarrow \frac{1}{w_n(\alpha)}.
\]
Thus we have
\[
\begin{aligned}
\lim_{\varepsilon\to0}\;\int_a^b\Phi\left(\frac{f_\varepsilon(x)}{k}\right) \mu(x) \dd x
&=
\lim_{\varepsilon\to0} \Phi\left(\frac{1}{k}\int_{\alpha-\varepsilon}^\alpha \frac{\nu[\alpha,b]}{\nu[a,\alpha]}\frac{1}{\varepsilon w_n(t)} \dd t\right) \mu[a,\alpha-\varepsilon]\\
&\quad\quad+
\Phi\left(\frac{1}{k}\int_\alpha^{\alpha+\varepsilon} \frac{1}{\varepsilon w_n(t)} \dd t\right)\mu[\alpha+\varepsilon, b]\\
&=
\Phi\left(\frac{\nu[\alpha,b]}{k\nu[a,\alpha]w_n(\alpha)}\right) \mu[a,\alpha]
+
\Phi\left(\frac{1}{k w_n(\alpha)}\right)\mu[\alpha, b].
\end{aligned}
\]
Thus,
\[
\begin{aligned}
\lim_{\varepsilon \to 0}
\left\|f_\varepsilon - \frac{1}{\nu[a,b]} \int_a^b f_\varepsilon \dd \nu \right\|_{L_\mu^\Phi[a,b]}
&=
\lim_{\varepsilon\to 0}\left\|f_\varepsilon\right\|_{L_\mu^\Phi[a,b]}\\
&=
\lim_{\varepsilon\to 0}\inf\left\{k > 0 : \int_a^b\Phi\left(\frac{f_\varepsilon}{k}\right) \dd \mu \leq 1\right\}\\
&= \inf\left\{k > 0 : \Phi\left(\frac{\nu[\alpha,b]}{k\nu[a,\alpha]w_n(\alpha)}\right) \mu[a,\alpha]
+
\Phi\left(\frac{1}{k w_n(\alpha)}\right)\mu[\alpha, b] \leq 1\right\}\\
&= \frac{1}{w_n(\alpha)} 
\left\|\frac{\nu[\alpha,b]}{\nu[a,\alpha]}\chi_{[a,\alpha]} + \chi_{[\alpha,b]}\right\|_{L_\mu^\Phi[a,b]},
\end{aligned}
\]
proving \autoref{eqn:p=1helper}.

Observe also that, because $\nu[a,\alpha] + \nu[\alpha,b] = \nu[a,b]$, we have
\[
\left\|
f_\varepsilon'
\right\|_{L_w^1[a,b]}
=
\frac{\nu[a,b]}{\nu[a,\alpha]}.
\]
This equality, together with \autoref{eqn:CW1=p} and \autoref{eqn:p=1helper}, show that
\[
\begin{aligned}
C &\geq \frac{\nu[a,\alpha]}{\nu[a,b] w_n(\alpha)}\left\|
\frac{\nu[\alpha,b]}{\nu[a,\alpha]}\chi_{[a,\alpha]} + \chi_{[\alpha,b]}
\right\|_{L_\mu^\Phi[a,b]}\\
&= \frac{1}{\nu[a,b]}
\frac{1}{w_n(\alpha)}
\left\|
\nu[\alpha,b]\chi_{[a,\alpha]} + \nu[a,\alpha]\chi_{[\alpha,b]}
\right\|_{L_\mu^\Phi[a,b]}\\
&\geq \frac{1}{\nu[a,b]}
\frac{1}{w_n(\alpha)}
\left\|
\nu[\alpha,b]\chi_{[a,\alpha]} - \nu[a,\alpha]\chi_{[\alpha,b]}
\right\|_{L_\mu^\Phi[a,b]},
\end{aligned}
\]
and since is true for all $a < \alpha < b$, we obtain that $C \geq K_{1,\Phi}(\mu,\nu,w_n)$. Since $w_n \geq w$ for all $n$, letting $n \to \infty$ we have
\[
K_{1,\Phi}(\mu,\nu,w) \leq C < \infty,
\]
completing the proof.
\end{proof}
\section{Sufficiency when \texorpdfstring{$p>1$}{p > 1}}\label{sec:p-suff}
For the rest of the paper, let $1 < p < \infty$. This section is devoted to the proof of \autoref{thm:CW1<p_forward}. We will adapt \cite[Propositions 1.6 \& 1.7]{chuawheeden}, which in turn use results from \cite{opickufner}. Note that the naming conventions of these two texts differ, and in this paper we follow \cite{chuawheeden} as much as possible.

Throughout this section, assume that $\Phi$ is submultiplicative and invertible on $[0,\infty)$, and that that the function $\Lambda$ defined by $\Lambda(t) = \Phi\left(t^{\frac{1}{p}}\right)$ is convex. Note that this implies, in particular, that $\Lambda$ is a Young function.
\begin{lemma}
\label{lemma:opickufneradaptation1}
Define
\[
S = S(p,\Phi, \mu, w) := 
\sup_{a<x<b} 
\left\{
\left[\Phi^{-1}\left(\frac{1}{\mu[x,b]^{1/2}}\right)\right]^{-2}
\left(\int_a^x w(t)^{1-p'}\dd t\right)^{1/p'} 
\right\}.
\]
If $S < \infty$, then
\[
\left\|\int_a^\bullet f(t)\dd t\right\|_{L_\mu^\Phi[a,b]}
\leq
C \left\|f\right\|_{L_w^p[a,b]}
\]
for all $f \in L_w^p[a,b]$ and for some constant $C \leq C_0(\Phi) S$, where $C_0(\Phi)$ is as in \autoref{thm:CW1<p_forward}.
\end{lemma}
\begin{proof}
The condition $S < \infty$ implies in particular that
\[
\int_a^t w(y)^{1-p'} \dd y < \infty
\]
for all $t \in (a,b)$. Consequently, $0 < h(t) <\infty$ for all $t \in (a,b)$, where
\[
h(t) := \left(\int_a^t w(y)^{1-p'} \dd y\right)^{1/2p'}.
\]
For all $x \in [a,b]$, we have, by H\"older's inequality,
\[
\begin{aligned}
\int_a^x f(t) \dd t &= \int_a^x \left(f(t)\;w(t)^{1/p} \;h(t)\right)\; \left(h(t)^{-1}\; w(t)^{-1/p}\right) \dd t \\
&\leq 
\left(\int_a^xf(t)^p\; h(t)^p\;w(t) \dd t\right)^{1/p}
\left(\int_a^x h(t)^{-p'}\; w(t)^{1-p'}\dd t\right)^{1/p'}
\end{aligned}
\]
Further, using the fundamental theorem of calculus, we obtain
\[
\begin{aligned}
\int_a^x h(t)^{-p'}\; w(t)^{1-p'}\dd t
&= 
\int_a^x \left(\int_a^t w(y)^{1-p'} \dd y\right)^{-1/2}\; w(t)^{1-p'}\dd t\\
&= 2 \left(\int_a^x w(y)^{1-p'}\dd y\right)^{1/2}\\
&= 2 h(x)^{p'}.
\end{aligned}
\]
Thus,
\[
\int_a^x f(t)\dd t 
\leq 
2^{1/p'}h(x)\left(\int_a^x f(t)^ph(t)^p w(t) \dd t\right)^{1/p},
\]
and so
\begin{equation}
\label{eqn:ok_ineq_1}
\left\|\int_a^\bullet f(t)\dd t\right\|_{L_\mu^\Phi[a,b]}^p
\leq
2^{p/p'}
\left\|
\left(
\int_a^\bullet f(t)^p\;h(t)^p w(t)\; \dd t
\right)^{1/p}
h\right\|_{L_\mu^\Phi[a,b]}^p.
\end{equation}
We claim that
\begin{equation}
\label{eqn:ok_ineq_2}
\left\|
\left(
\int_a^\bullet f(t)^p\;h(t)^pw(t)\; \dd t
\right)^{1/p}
h \right\|_{L_\mu^\Phi[a,b]}^p
\leq
2\int_a^bf(t)^p\;h(t)^p\; w(t)\;
\left\|
h
\right\|_{L_\mu^\Phi[t,b]}^p\dd t.
\end{equation}
To show this, we will use the assumption that $\Lambda(t) = \Phi\left(t^{\frac{1}{p}}\right)$ is convex which, together with the properties of $\Phi$, implies that $\Lambda$ is a Young function. The reader may verify that $\|f\|_{L_\mu^\Phi[a,b]}^p = \|f^p\|_{L_\mu^\Lambda[a,b]}$. Together with \autoref{lemma:gen_minkowski}, this gives
\[
\begin{aligned}
\left\|\left(\int_a^\bullet f(t)^p\;h(t)^pw(t)\; \dd t\right)^{1/p}
h \right\|_{L_\mu^\Phi[a,b]}^p
&=
\left\|\left(\int_a^\bullet f(t)^p\;h(t)^pw(t)\; \dd t\right)
h^p \right\|_{L_\mu^\Lambda[a,b]}\\
&=
\left\|\left(\int_a^b f(t)^p\;h(t)^p \chi_{[a,\bullet]}(t) w(t)\; \dd t\right)
h^p \right\|_{L_\mu^\Lambda[a,b]}\\
&\leq 
2\int_a^b f(t)^p\;h(t)^p w(t) \left\|\chi_{[t,b]}(\bullet)h^p \right\|_{L_\mu^\Lambda[a,b]} \dd t\\
&=
2\int_a^b f(t)^p\;h(t)^p w(t) \left\|h \right\|_{L_\mu^\Phi[t,b]}^p \dd t,
\end{aligned}
\]
proving the claim. Combining \autoref{eqn:ok_ineq_1} and \autoref{eqn:ok_ineq_2}, we have
\begin{equation}
\label{eqn:ok_ineq_3} 
\left\|\int_a^\bullet f(t)\dd t\right\|_{L_\mu^\Phi[a,b]}^p
\leq
2^{1+(p/p')}
\int_a^bf(t)^p\;h(t)^p\; w(t)\;
\left\|
h
\right\|_{L_\mu^\Phi[t,b]}^p\dd t.
\end{equation}
By definition of $S$,
\begin{equation}
\label{eqn:ok_ineq_4}
\left\|h\right\|_{L_\mu^\Phi[t,b]} 
=
\left\|\left(\int_a^\bullet w(y)^{1-p'} \dd y\right)^{1/2p'}\right\|_{L_\mu^\Phi[t,b]} 
\leq 
S^{1/2}\left\|
\Phi^{-1}\left(\frac{1}{\mu[\bullet,b]^{1/2}}\right)\right\|_{L_\mu^\Phi[t,b]}.
\end{equation}
By definition of the gauge norm,
\[
\left\|
\Phi^{-1}\left(\frac{1}{\mu[\bullet,b]^{1/2}}\right)\right\|_{L_\mu^\Phi[t,b]}
=
\inf\left\{k > 0 : \int_t^b\Phi\left(
\frac{1}{k} \Phi^{-1}\left(\frac{1}{\mu[x,b]^{1/2}}\right)
\right) \; \mu(x)\dd x \leq 1\right\}
\]
Using submultiplicativity of $\Phi$, we have
\[
\begin{aligned}
\int_t^b\Phi\left(
\frac{1}{k}\Phi^{-1}\left(\frac{1}{\mu[x,b]^{1/2}}\right)
\right) \; \mu(x)\dd x
&\leq
\Phi\left(
\frac{1}{k}
\right)
\int_t^b\frac{1}{\mu[x,b]^{1/2}} \; \mu(x)\dd x\\
&=\Phi\left(
\frac{1}{k}
\right)
\int_0^{\mu[t,b]} \frac{1}{z^{1/2}}\dd z\\
&= 2\;\Phi\left(
\frac{1}{k}
\right)
\mu[t,b]^{1/2}
\end{aligned}
\]
Setting $2\;\Phi\left(\frac{1}{k}\right)\mu[t,b]^{1/2} = 1$, we obtain
\[
\begin{aligned}
\frac{1}{k} &= \Phi^{-1}\left(\frac{1}{2} \mu[t,b]^{-1/2}\right)\\
&\geq 
\Phi^{-1}\left(\frac{1}{2}\right)
\Phi^{-1}\left(\mu[t,b]^{-1/2}\right)
\end{aligned}
\]
Note that the ultimate inequality uses the supermultiplicativity of $\Phi^{-1}$, which in turn follows from the submultiplicativity of $\Phi$. Thus, we have that
\[
\begin{aligned}
\left\|
\Phi^{-1}\left(\frac{1}{\mu[\bullet,b]^{1/2}}\right)\right\|_{L_\mu^\Phi[t,b]} 
&\leq k\\
&\leq \left[\Phi^{-1}\left(\frac{1}{2}\right)
\Phi^{-1}\left(\mu[t,b]^{-1/2}\right)\right]^{-1}\\
&\leq
\left[\Phi^{-1}\left(\frac{1}{2}\right)\right]^{-1} S^{1/2}
\left[\left(\int_a^t w(y)^{1-p'}\dd y\right)^{1/p'}\right]^{-1/2}\\
&= \left[\Phi^{-1}\left(\frac{1}{2}\right)\right]^{-1} S^{1/2}
\left(h(t)^2\right)^{-1/2}\\
&= \left[\Phi^{-1}\left(\frac{1}{2}\right)\right]^{-1} S^{1/2}
\frac{1}{h(t)}.
\end{aligned}
\]
Together with \autoref{eqn:ok_ineq_3} and \autoref{eqn:ok_ineq_4}, this gives
\[
\begin{aligned}
\left\|\int_a^\bullet f(t)\dd t\right\|_{L_\mu^\Phi[a,b]}^p
&\leq
2^{1+(p/p')}
\int_a^bf(t)^p\;h(t)^p\; w(t)\;
\left(S^{1/2}\;\left[\Phi^{-1}\left(\frac{1}{2}\right)\right]^{-1} S^{1/2}\;\frac{1}{h(t)}\right)^p\dd t\\
&=
2^{1+(p/p')}
\int_a^bf(t)^p\;h(t)^p\; w(t)\;
S^{p/2}\;\left[\Phi^{-1}\left(\frac{1}{2}\right)\right]^{-p} S^{p/2}\;\frac{1}{h(t)^p}\dd t\\
&=
2^{1+(p/p')}
\left[\Phi^{-1}\left(\frac{1}{2}\right)\right]^{-p}S^p \;
\int_a^b f(t)^p\; w(t)^p\dd t.
\end{aligned}
\]
Taking $p$th roots, we have
\[
\begin{aligned}
\left\|\int_a^\bullet f(t)\dd t\right\|_{L_\mu^\Phi[a,b]}
&\leq
2
\left[\Phi^{-1}\left(\frac{1}{2}\right)\right]^{-1}
\; S \;\left(\int_a^b f(t)^p\; w(t)^p\dd t\right)^{1/p}\\
&=
C_0(\Phi)
S \;
\|f\|_{L_w^p[a,b]},
\end{aligned}
\]
as desired.
\end{proof}
\begin{lemma}
Define
\[
T = T(p,\Phi, \mu, w) := 
\sup_{a<x<b} 
\left\{
\left[\Phi^{-1}\left(\frac{1}{\mu[a,x]^{1/2}}\right)\right]^{-2}
\left(\int_x^b w^{1-p'}(t)\dd t\right)^{1/p'} 
\right\}.
\]
If $T< \infty$, then
\[
\left\|\int_\bullet^b f(t)\dd t\right\|_{L_\mu^\Phi[a,b]}
\leq
C \left\|f\right\|_{L_w^p[a,b]},
\]
for all $f \in L_w^p[a,b]$ and for some constant $C \leq C_0(\Phi) T$, where $C_0(\Phi)$ is as in \autoref{thm:CW1<p_forward}.
\end{lemma}
\begin{proof}
The proof of this inequality is nearly identical to that of \autoref{lemma:opickufneradaptation1}, and is left as an exercise to the reader.
\end{proof}
The previous two lemmas lead to the following corollary, which is analogous to (one direction of) \cite[Propositions 1.6 \& 1.7]{chuawheeden}.
\noindent
\begin{corollary}
\label{cor:CWintermediate}
Let $\tau$ be a weight on $[a,b]$, and let $C_0(\Phi)$ be as in the statement of \autoref{thm:CW1<p_forward}.
\begin{enumerate}[label=(\roman*)]
\item Define
\[
S(p,\Phi,\mu,\tau,w) = 
\sup_{a<x<b} 
\left\{
\left[\Phi^{-1}\left(\frac{1}{\mu[x,b]^{1/2}}\right)\right]^{-2}
\left(\int_a^x \tau(t)^{p'}\;w(t)^{1-p'}\dd t\right)^{1/p'}
\right\}.
\]
If $S(p,\Phi, \mu, \tau, w) < \infty$, then
\[
\left\|\int_a^\bullet f(t)\tau(t)\dd t\right\|_{L_\mu^\Phi[a,b]}
\leq
C\left\|f\right\|_{L_w^p[a,b]}
\]
for all $f \in L_w^p[a,b]$ and for some constant $C \leq C_0(\Phi) \;S(p,\Phi, \mu, \tau, w)$.
\item Define
\[
T(p,\Phi, \mu, \tau, w) := 
\sup_{a<x<b} 
\left\{
\left[\Phi^{-1}\left(\frac{1}{\mu[a,x]^{1/2}}\right)\right]^{-2}
\left(\int_x^b \tau(t)^{p'}\;w(t)^{1-p'}\dd t\right)^{1/p'}
\right\}.
\]
If $T(p,\Phi, \mu, \tau, w) < \infty$, then
\[
\left\|\int_\bullet^b f(t)\tau(t)\dd t\right\|_{L_\mu^\Phi[a,b]}
\leq
C\left\|f\right\|_{L_w^p[a,b]},
\]
for all $f \in L_w^p[a,b]$ and for some constant $C \leq C_0(\Phi) \; T(p,\Phi, \mu, \tau, w)$.
\end{enumerate}
\end{corollary}
\begin{proof}
The statements (i) and (ii) follow from the previous two lemmas, by replacing $f$ and $w$ by $f\tau$ and $w/\tau^p$ respectively.
\end{proof}
Finally, we are ready to prove \autoref{thm:CW1<p_forward}.
\begin{proof}[Proof of \autoref{thm:CW1<p_forward}]
Observe,
\[
\begin{aligned}
\left\|f - \frac{1}{\nu[a,b]}\int_a^b f \;dv\right\|_{L_\mu^\Phi[a,b]} 
&=\left\|
\frac{1}{\nu[a,b]}
\int_a^\bullet \nu[a,z]^{p'} f'(z) \dd z
+
\int_\bullet^b \nu[z,b]^{p'} f'(z) \dd z\right\|_{L_\mu^\Phi[a,b]}\\
&\leq 
\frac{1}{\nu[a,b]}
\left[
\left\|\int_a^\bullet \nu[a,z]^{p'} f'(z) \dd z\right\|_{L_\mu^\Phi[a,b]}
+
\left\|\int_\bullet^b \nu[z,b]^{p'} f'(z) \dd z\right\|_{L_\mu^\Phi[a,b]}
\right]\\
&\leq
\frac{C_0(\Phi)}{\nu[a,b]}
\left(S(p,\Phi, \mu, \nu[a,\bullet], w) + T(p,\Phi, \mu, \nu[\bullet,b], w)\right)\|f'\|_{L_w^p[a,b]}\\
&=
C_0(\Phi) K_{p,\Phi}(\mu,\nu, w)\|f'\|_{L_w^p[a,b]},
\end{aligned}
\]
where the penultimate inequality uses \autoref{cor:CWintermediate}.
\end{proof}
\section{Necessity when \texorpdfstring{$p>1$}{p > 1}}\label{sec:p-necc}
Throughout this section, assume that $\Phi$ is invertible on $[0,\infty)$.

\begin{proof}[Proof of \autoref{thm:CW1<p_backward}]
Assume that for all Lipschitz continuous functions $f$, there exists some constant $C > 0$ such that
\[
\left\|f - \frac{1}{\nu[a,b]}\int_a^b f \dd \nu\right\|_{L_\mu^\Phi[a,b]}
\leq
C 
\left\|f'\right\|_{L_w^p[a,b]}.
\]
Fix $\alpha \in (a,b)$, and define
\[
\begin{aligned}
f_1(x) &= \int_a^x \nu[a,t]^{p'-1}w_n(t)^{1-p'} \chi_{[a,\alpha]}(t) \dd t,\\
f_2(x) &= \int_a^x \nu[t,b]^{p'-1}w_n(t)^{1-p'} \chi_{[\alpha,b]}(t) \dd t.
\end{aligned}
\]
Note that
\[
\begin{aligned}
\left\|f_1 - \frac{1}{\nu[a,b]}\int_a^b f_1 \dd \nu\right\|_{L_\mu^\Phi[a,b]}
&= 
\left\|
\frac{1}{\nu[a,b]}\left(\int_a^\bullet 
\nu[a,z]f_1'(z) \dd z - \int_\bullet^b \nu[z,b] f_1'(z) \dd z
\right)\right\|_{L_\mu^\Phi[a,b]}\\
&\geq
\left\|
\frac{1}{\nu[a,b]}\left(\int_a^\bullet 
\nu[a,z]f_1'(z) \dd z - \int_\bullet^b \nu[z,b] f_1'(z) \dd z
\right)\right\|_{L_\mu^\Phi[\alpha,b]}\\
&=
\left\|
\frac{1}{\nu[a,b]}\int_a^\bullet 
\nu[a,z]f_1'(z) \dd z
\right\|_{L_\mu^\Phi[\alpha,b]}\\
&=
\left\|
\frac{1}{\nu[a,b]}\int_a^\bullet 
\nu[a,z]\nu[a,z]^{p'-1}w_n(z)^{1-p'} \chi_{[a,\alpha]}(z) \dd z
\right\|_{L_\mu^\Phi[\alpha,b]}\\
&=
\left\|
\frac{1}{\nu[a,b]}\int_a^\alpha
\nu[a,z]^{p'}w_n(z)^{1-p'} \dd z
\right\|_{L_\mu^\Phi[\alpha,b]}\\
&=
\frac{1}{\nu[a,b]}\left(\int_a^\alpha
\nu[a,z]^{p'}w_n(z)^{1-p'} \dd z\right)
\left\|1\right\|_{L_\mu^\Phi[\alpha,b]}\\
&=
\frac{1}{\nu[a,b]}\left(\int_a^\alpha
\nu[a,z]^{p'}w_n(z)^{1-p'} \dd z\right)
\left[\Phi^{-1}\left(\frac{1}{\mu[\alpha,b]}\right)\right]^{-1}
\end{aligned}
\]
On the other hand,
\[
\left\|f_1'\right\|_{L_{w_n}^p[a,b]}^p
=
\int_a^b |f_1'(z)|^p w_n(z) \dd z
= 
\int_a^\alpha \nu[a,z]^{p'} w_n(z)^{1-p'} \dd z
\]
by definition of $f_1$. Hence, by the Poincaré inequality,
\[
\frac{1}{\nu[a,b]}\left(\int_a^\alpha
\nu[a,z]^{p'}w_n(z)^{1-p'} \dd z\right)
\left[\Phi^{-1}\left(\frac{1}{\mu[\alpha,b]}\right)\right]^{-1}
\leq
C
\left(\int_a^\alpha \nu[a,z]^{p'} w_n(z)^{1-p'} \dd z\right)^{1/p}.
\]
This implies
\[
\frac{1}{\nu[a,b]}\left(\int_a^\alpha
\nu[a,z]^{p'}w_n(z)^{1-p'} \dd z\right)^{1/p
'}
\left[\Phi^{-1}\left(\frac{1}{\mu[\alpha,b]}\right)\right]^{-1}
\leq
C.
\]
Letting $n \to \infty$, we see that, by the monotone convergence theorem, the above holds with $w_n$ replaced by $w$.

Applying a similar argument to $f_2$, we have
\[
\frac{1}{\nu[a,b]}\left(\int_\alpha^b
\nu[z,b]^{p'}w(z)^{1-p'} \dd z\right)^{1/p
'}
\left[\Phi^{-1}\left(\frac{1}{\mu[a,\alpha]}\right)\right]^{-1}
\leq
C.
\]
Taking the supremum over all $a < \alpha < b$, we see that $\tilde{K}_{p,\Phi}(\mu,\nu,w) \leq 2C$, i.e.
\[
\frac{1}{2}\tilde{K}_{p,\Phi}(\mu,\nu,w)
\leq C,
\]
completing the proof.
\end{proof}
\section{Example}
Let $0 = a < b < \infty$, $1 < p < \infty$, and let $\nu = \mu = w$, where
\[
w(x)=\frac{2e^{-\frac{1}{x^2}}}{x^3}.
\]
Note that $w$ is infinitely degenerate at the origin, and that
\[
\int_{0}^{x}w(t)dt=e^{-\frac{1}{x^2}}.
\]
In this section, we show that this choice of $w$ is an example of a weight that satisfies the $(\Phi, p)$ Poincar\'e inequality, but not the $(q,p)$ Poincar\'e inequality for any $q > p$.

We first show that the constant
\begin{equation}
\label{eqn:workedexample_K_pp}
\begin{split}
K_{p,p}(w,w,w)=&\frac{1}{w[0,b]}\left(\sup_{0<x<b}\left[w[x,b]^{1/p}\left(\int_{0}^{x}w[0,t]^{p'}w(t)^{1-p'} \dd t\right)^{1/p'}\right]\right.\\
&+\sup_{0<x<b}\left.\left[w[0,x]^{1/p}\left(\int_{x}^{b}w[t,b]^{p'}w(t)^{1-p'} \dd t\right)^{1/p'}\right]\right)
\end{split}
\end{equation}
is finite, implying, by \cite[Theorem 1.4]{chuawheeden}, that the $(p,p)$ Poincar\'e inequality holds for all Lipschitz continuous functions $f \colon [0,b] \to \R$. Observe that, for all $x \in (0,b)$, we have
\[
\begin{aligned}
&w[x,b]^{1/p}\left(\int_{0}^{x}w[0,t]^{p'}w(t)^{1-p'} \dd t\right)^{1/p'}\\
=\;&\left(e^{-\frac{1}{b^2}} - e^{-1/x^2}\right)^{1/p}
\left(\int_0^x e^{-p'/t^2 - 1/t^2 + p'/t^2} \left(\frac{2}{t^3}\right)^{1-p'}\dd t\right)^{1/p'}\\
=\;&2^{1-p'}\left(e^{-\frac{1}{b^2}} - e^{-1/x^2}\right)^{1/p}
\left(\int_0^x e^{-1/t^2} t^{3(p'-1)}\dd t\right)^{1/p'}\\
\leq\;& 2^{1-p'}e^{-1/pb^2}
\left(\int_0^b e^{-1/t^2} t^{3(p'-1)}\dd t\right)^{1/p'}\\
<\;&\infty,
\end{aligned}
\]
since the last integrand is increasing on $[0,b]$, and hence bounded. As such, the first supremum in \autoref{eqn:workedexample_K_pp} is finite. Also observe that
\[
\begin{aligned}
&w[0,x]^{1/p}\left(\int_{x}^{b}w[t,b]^{p'}w(t)^{1-p'} \dd t\right)^{1/p'}\\
=\;&e^{-\frac{1}{px^2}}
\left(\int_x^b \left(e^{-\frac{1}{b^2}-\frac{1}{p't^2}+1/t^2} - e^{-\frac{1}{p't^2}}\right)^{p'}\left(\frac{2}{t^3}\right)^{1-p'}\dd t\right)^{1/p'}\\
=\;&2^{1-p'}e^{-\frac{1}{px^2}}
\left(\int_x^b \left(e^{-\frac{1}{b^2} + \frac{1}{pt^2}} - e^{-\frac{1}{p't^2}}\right)^{p'}t^{3(p'-1)}\dd t\right)^{1/p'}\\
\leq\;& e^{-\frac{1}{px^2}}
\left(\int_x^b \left(e^{\frac{1}{pt^2}}\right)^{p'}t^{3(p'-1)}\dd t\right)^{1/p'}\\
\end{aligned}
\]
For $t \in [x,b]$, we see that
\[
e^{-\frac{1}{px^2} + \frac{1}{pt^2}} \leq e^{-\frac{1}{px^2} + \frac{1}{px^2}} = 1,
\]
and hence the second supremum in \autoref{eqn:workedexample_K_pp} is also finite. Thus $K_{p,p}(w,w,w) < \infty$, yielding the $(p,p)$ Poincar\'e inequality.

Next, we consider $K_{p,q}(w,w,w)$ for some $q > p$. The first supremum in the constant is finite by a similar argument as above; however, we are less fortunate with the second supremum. Observe that
\[
\begin{aligned}
&w[0,x]^{1/q}\left(\int_x^b w[t,b]^{p'}w(t)^{1-p'} \dd t\right)^{1/p'}\\
=\;&2^{1-p'}e^{-\frac{1}{qx^2}}
\left(\int_x^b \left(e^{-\frac{1}{b^2} + \frac{1}{pt^2}} - e^{-\frac{1}{p't^2}}\right)^{p'}t^{3(p'-1)}\dd t\right)^{1/p'}\\
=\;&2^{1-p'}
\left(\int_x^b \left(\left(e^{-\frac{1}{b^2}} - e^{-\frac{1}{t^2}}\right)e^{-\frac{1}{qx^2} + \frac{1}{pt^2}}\right)^{p'}t^{3(p'-1)}\dd t\right)^{1/p'}\\
\end{aligned}
\]
Let $\varepsilon \in \R^+$ be such that $\frac{1}{q} < \varepsilon < \frac{1}{p}$, and choose $t_0 \in (x,b)$ such that $\frac{1}{pt_0^2} = \frac{\varepsilon}{x^2}$. Then $t_0 = \sqrt{\frac{x^2}{\varepsilon p}}$, and so
\[
\int_x^b \left(\left(e^{-\frac{1}{b^2}} - e^{-\frac{1}{t^2}}\right)e^{-\frac{1}{qx^2} + \frac{1}{pt^2}}\right)^{p'}t^{3(p'-1)}\dd t 
\geq 
\int_x^{\sqrt{\frac{x^2}{\varepsilon p}}}
\left(\left(e^{-\frac{1}{b^2}} - e^{-\frac{1}{t^2}}\right)e^{-\frac{1}{qx^2} + \frac{1}{pt^2}}\right)^{p'}t^{3(p'-1)}\dd t
\]
Observe that $g(t) := \left(e^{-\frac{1}{b^2}} - e^{-\frac{1}{t^2}}\right)e^{-\frac{1}{qx^2} + \frac{1}{pt^2}}$ is a decreasing function of $t$, and hence for $t \in \left(x, \sqrt{\frac{x^2}{\varepsilon p}}\right)$
\[
\begin{aligned}
g(t) &\geq g(t_0)\\
&= \left(e^{-\frac{1}{b^2}} - e^{-\frac{\varepsilon p}{x^2}}\right)e^{\frac{1}{x^2}\left(\varepsilon - \frac{1}{q}\right)}
\end{aligned}
\]
where the second factor of the righthand side blows up as $x \to 0$, since $\varepsilon - \frac{1}{q} > 0$. Combining this with the above, we obtain
\[
\begin{aligned}
&w[0,x]^{1/q}\left(\int_x^b w[t,b]^{p'}w(t)^{1-p'} \dd t\right)^{1/p'}\\
\geq\;& 2^{1-p'}g(t_0) \left(\int_x^{t_0} t^{3(p'-1)} \dd t\right)^{1/p'}.
\end{aligned}
\]
Since the integral above is positive, the quantity blows up as $x \to 0$ because of the factor of $g(t_0)$. That is, $K_{p,q}(w,w,w)$ is not finite, so the $(q,p)$ Poincar\'e inequality does not hold.

Finally, we show that $K_{p, \Phi}(w,w,w) < \infty$, for an appropriate choice of $\Phi$ depending on $p$. Let $\alpha \in \left(0,\frac{3}{4}p\right)$, and define
\[
    \Phi(t) =
    \begin{cases}
    |t|^p\left(\ln|t|\right)^\alpha, &\text{if $|t| \geq e^{2\alpha};$}\\
    |t|^p\left(2\alpha\right)^\alpha, &\text{if $|t| < e^{2\alpha}.$}
    \end{cases}
\]
Similar to \cite[Conclusion 26]{KRSSh} it can be shown that $\Phi$ is a submultiplicative piecewise differentiable convex function that vanishes at $0$. In particular, convexity follows from the fact that
\[
\begin{aligned}
\frac{\dd}{\dd t}
\left(|t|^p\left(\ln|t|\right)^\alpha\right)\Bigg\vert_{t=e^{2\alpha}}
&=
\left(e^{2\alpha}\right)^{p-1}
\cdot \left(2\alpha\right)^\alpha 
\cdot \left(p + \frac{1}{2}\right)\\
&\geq
\left(e^{2\alpha}\right)^{p-1}
\cdot \left(2\alpha\right)^\alpha\\
&=
\frac{\dd}{\dd t}
\left(|t|^p\left(2\alpha\right)^\alpha\right)\Bigg\vert_{t=e^{2\alpha}}.
\end{aligned}
\]
Moreover,
\[
\Lambda(t) = \Phi\left(|t|^{\frac{1}{p}}\right) = 
\begin{cases}
    |t|\left(\frac{1}{p}\ln|t|\right)^\alpha, &\text{if $|t| \geq e^{2\alpha p};$}\\
    |t|\left(2\alpha\right)^\alpha, &\text{if $|t| < e^{2\alpha p}.$}
\end{cases}
\]
is convex, and is therefore a Young function. Since the assumptions of \autoref{thm:CW1<p_forward} are satisfied, we may show that the $(\Phi, p)$ Poincar\'e inequality holds for this choice of $\Phi$ by showing that $K_{p, \Phi}(w,w,w) < \infty$.

For large enough $t$, note that
\[
\Phi^{-1}(t) \approx \left(\frac{t}{\ln(t)^\alpha}\right)^{1/p}.
\]
Consider the constant $K_{p,\Phi}(w,w,w)$ as defined in \autoref{thm:CW1<p_forward}. The first supremum is finite as in the previous cases, so it remains to examine the second supremum. Note that
\[
\begin{aligned}
\Phi^{-1}\left(\frac{1}{w[0,x]^{1/2}}\right)
&= \Phi^{-1}\left(e^{\frac{1}{2x^2}}\right)\\
&\approx \left(x^{2\alpha}e^{\frac{1}{2x^2}}\right)^{1/p}\\
&= x^{\frac{2\alpha}{p}}e^{\frac{1}{2x^2p}}
\end{aligned}
\]
and so
\[
\left[\Phi^{-1}\left(\frac{1}{w[0,x]^{1/2}}\right)\right]^{-2}
\approx
x^{-\frac{4\alpha}{p}}e^{-\frac{1}{px^2}}
\]
Hence,
\[
\begin{aligned}
&\left[\Phi^{-1}\left(\frac{1}{w[0,x]^{1/2}}\right)\right]^{-2}\biggl(\int_x^b w[t,b]^{p'} w(t)^{1-p'} \dd t\biggr)^{1/p'}\\
\approx\;&
x^{-\frac{4\alpha}{p}}e^{-\frac{1}{px^2}}
\left(\int_x^b \left(e^{-\frac{1}{b^2} + \frac{1}{pt^2}} - e^{-\frac{1}{p't^2}}\right)^{p'}t^{3(p'-1)}\dd t\right)^{1/p'}\\
=\;&e^{-\frac{1}{b^2}}\left(\int_{x}^{b}\left(e^{\frac{1}{pt^2}-\frac{1}{px^2}}-e^{\frac{1}{b^2}-\frac{1}{p't^2}-\frac{1}{px^2}}\right)^{p'}x^{-4\alpha(p'-1)}t^{3(p'-1)}dt\right)^{1/p'}.
\end{aligned}
\]
 The term we need to estimate is therefore
 \[
\int_{x}^{b} e^{-(p'-1)\left(\frac{1}{x^2}-\frac{1}{t^2}\right)}x^{-4\alpha(p'-1)}t^{3(p'-1)}dt
=
C\int_{x}^{b} e^{-(p'-1)\left(\frac{1}{x^2}-\frac{1}{t^2}\right)}x^{-4\alpha(p'-1)}t^{3p'}d\left(-\frac{1}{t^2}\right).
 \]
 We split the integral into two,
 \[
\int_{x}^{b} e^{-(p'-1)\left(\frac{1}{x^2}-\frac{1}{t^2}\right)}x^{-4\alpha(p'-1)}t^{3p'}d\left(-\frac{1}{t^2}\right)
=
\int_{x}^{\delta(x)} +\int_{\delta(x)}^{b} e^{-(p'-1)\left(\frac{1}{x^2}-\frac{1}{t^2}\right)}x^{-4\alpha(p'-1)}t^{3p'}d\left(-\frac{1}{t^2}\right)=:I+II,
 \]
 where $\delta(x)>x$ will be determined later. For the first term we have
 \[
 I\leq x^{-4\alpha(p'-1)}\delta(x)^{3p'}\int_{x}^{\delta(x)}e^{-(p'-1)\left(\frac{1}{x^2}-\frac{1}{t^2}\right)}d\left(-\frac{1}{t^2}\right).
 \]
 We now choose $\delta(x)$ so that $\delta(x)^{3p'}=x^{4\alpha(p'-1)}$, i.e.
 \[
 \delta(x)=x^{\frac{4\alpha}{3p}}.
 \]
 Note that $\delta(x) > x$ if $x < 1$, since $\alpha < \frac{3}{4}p$. This gives
 \[
  I\leq \int_{x}^{\delta(x)}e^{-(p'-1)\left(\frac{1}{x^2}-\frac{1}{t^2}\right)}d\left(-\frac{1}{t^2}\right)
  = Ce^{-(p'-1)\left(\frac{1}{x^2}-\frac{1}{t^2}\right)}\Big|_{t=x}^{\delta(x)}
  = -C\left(1-e^{-(p'-1)\left(\frac{1}{x^2}-\frac{1}{x^{\frac{8\alpha}{3p}}}\right)}\right)
  \to C,
 \]
 as $x \to 0$, since $\frac{1}{x^2}-\frac{1}{x^{\frac{8\alpha}{3p}}}>0$.
 For II, we have
 \[
     \begin{aligned}
 II &= 
 \int_{\delta(x)}^{b} e^{-4\alpha(p'-1)\left(\frac{1}{x^2}-\frac{1}{t^2}\right)}x^{-(p'-1)}t^{3(p'-1)}dt\\
 &\leq x^{-4\alpha(p'-1)}e^{-(p'-1)\left(\frac{1}{x^2}-\frac{1}{x^{8\alpha/(3p)}}\right)}\int_{\delta(x)}^{b}t^{3(p'-1)}dt\\
 &\leq Cx^{-4\alpha(p'-1)}e^{-(p'-1)\left(\frac{1}{x^2}-\frac{1}{x^{8\alpha/(3p)}}\right)}\to 0,
\end{aligned}
 \]
 as $x \to 0$. So, the $(\Phi, p)$ Poincar\'e inequality indeed holds for this choice of $w$.
\newpage
\bibliography{biblio}
\bibliographystyle{abbrv}
\nocite{*}
\vspace{+0.25in}
\end{document}